\documentclass{article}
\usepackage[utf8]{inputenc}
\usepackage{amsmath,amssymb,amsthm,amsfonts,latexsym,graphicx,subfigure,mathtools}

\usepackage{url}
\usepackage{xcolor}

\newcommand{\R}{\mathbb{R}}
\DeclareMathOperator{\inp}{in}
\DeclareMathOperator{\out}{out}
\DeclareMathOperator{\id}{Id}

\theoremstyle{plain}
\newtheorem{theorem}{Theorem}[section]
\newtheorem{lemma}[theorem]{Lemma}
\newtheorem{proposition}[theorem]{Proposition}
\newtheorem{corollary}[theorem]{Corollary}
\newtheorem{conjecture}[theorem]{Conjecture}

\theoremstyle{definition}
\newtheorem{definition}[theorem]{Definition}
\newtheorem{example}[theorem]{Example}
\theoremstyle{remark}
\newtheorem{remark}[theorem]{Remark}

\usepackage[affil-it]{authblk}
\usepackage{booktabs}   
\usepackage[numbers,sort&compress]{natbib}

\usepackage[margin=3.5cm]{geometry}

\title{There is no going back:\\ Properties of the non-backtracking Laplacian}

\author[1,2]{Raffaella Mulas\thanks{r.mulas@vu.nl}}
\author[3]{Dong Zhang}
\author[2]{Giulio Zucal}
\affil[1]{Vrije Universiteit Amsterdam, Amsterdam, The Netherlands}
\affil[2]{Max Planck Institute for Mathematics in the Sciences, Leipzig, Germany}
\affil[3]{LMAM and School of Mathematical Sciences, Peking University, Beijing, China}

\date{}

\allowdisplaybreaks

\begin{document}

\maketitle
\begin{abstract} We prove new properties of the non-backtracking graph and the non-backtracking Laplacian for graphs. In particular, among other results, we prove that two simple graphs are isomorphic if and only if their corresponding non-backtracking graphs are isomorphic, and we investigate properties of various classes of non-backtracking Laplacian eigenfunctions, such as symmetric and antisymmetric eigenfunctions. Moreover, we introduce and study circularly partite graphs as a generalization of bipartite graphs, and we use this notion to state a sharp upper bound for the spectral gap from $1$. We also investigate the singular values of the non-backtracking Laplacian in relation to independence numbers, and we use them to bound the moduli of the eigenvalues.

\vspace{0.2cm}
\end{abstract}

\section{Introduction}
Non-backtracking walks and non-backtracking operators have been studied for more than three decades, and in various areas of graph theory. Hashimoto introduced the non-backtracking matrix in 1989, in the context of graph Zeta functions \cite{hashimoto1989zeta}, and as a continuation of his work, many results on this matrix have been proved in the area of algebraic graph theory \cite{terras2010zeta,bass1992ihara,cooper2009}. In 2015, Backhausz, Szegedy and Virág applied non-backtracking random walks to the study of Ramanujan graphings \cite{backhausz2015ramanujan}. Shrestha, Scarpino, and Moore in 2015 \cite{shrestha2015message}, then Castellano and Pastor-Satorras in 2018 \cite{castellano2018relevance}, applied the study of non-backtracking paths to epidemic spreading on networks. Moreover, the non-backtracking matrix has been shown to be very powerful in spectral graph theory, as well as in its applications to network analysis. We refer also to \cite{torres2021nonbacktracking,Leo2007,torres2020node,torres2019non,coste2021eigenvalues,glover2021some,pastor2020localization,bordenave2018nonbacktracking,krzakala2013spectral,martin2014localization,Arrigo1,Arrigo2,Gronford1,Mellor1,Godsil16,Huang19} for a vast\,---\,but not complete\,---\,literature on this topic.

Recently, Jost, Mulas and Torres \cite{NB-Laplacian} introduced the non-backtracking Laplacian and showed that its eigenvalues encode several structural properties of the graph in a more precise way than other known graph operators such as the adjacency matrix, the symmetric Laplacian matrices and the non-backtracking matrix. As a continuation of \cite{NB-Laplacian}, here we investigate new properties of the non-backtracking graph and the non-backtracking Laplacian. Before giving an overview of known and new properties, we give the basic notations and definitions that will be needed throughout the paper, following mainly \cite{NB-Laplacian}. \newline


Fix a \emph{simple graph} $G=(V,E)$, that is, an undirected, unweighted graph without multi-edges and without loops. Assume that $G$ has $N$ nodes and $M$ edges. We define the \emph{degree} of a vertex $v$, denoted $\deg_G v$ or simply $\deg v$, as the number of edges in which it is contained, and we assume that all vertices have degree $\geq 1$, that is, there are no isolated vertices. If two vertices $v,w\in V$ are connected by an edge, we write $v\sim w$, or equivalently $w\sim v$, and we denote such edge by $(v,w)$, or equivalently by $(w,v)$.\newline
Choosing an \emph{orientation} for an edge means letting one of its endpoints be its \emph{input} and the other one be its \emph{output}. We let $e=[v, w]$ denote the oriented edge whose input is $v$ and whose output is $w$. In this case, we write $\inp(e) \coloneqq v$ and $\out(e) \coloneqq w$. Moreover, we let $e^{-1} \coloneqq [w, v]$.\newline
From here on, we also fix an arbitrary orientation for each edge. We let $e_1,\ldots,e_M$ denote the edges of $G$ with this fixed orientation, and we let 
$$e_{M+1} \coloneqq e^{-1}_1,\ldots,e_{2M} \coloneqq e_M^{-1}$$
denote the edges with the inverse orientation.

\begin{definition}
A \emph{non-backtracking random walk} on $G$ is a discrete-time Markov process on the oriented edges such that the probability of going from $e_i$ to $e_j$ is
\begin{equation*}
    \mathbb{P}(e_i\rightarrow e_j)=\begin{cases}\frac{1}{\deg (\out(e_i))-1} & \text{if }\out(e_i)=\inp(e_j) \text{ and }  \inp(e_i)\neq\out(e_j)\\
    0 & \text{otherwise.}
    \end{cases}
\end{equation*}
\end{definition}
We now define the non-backtracking matrix of $G$, as follows.
\begin{definition}
The matrix $B\coloneqq B(G)$ is the $2M\times 2M$ matrix with $(0,1)$--entries such that
\begin{equation*}
    B_{ij}=1\iff \out(e_i)=\inp(e_j) \text{ and }  \inp(e_i)\neq\out(e_j).
\end{equation*}
The \emph{non-backtracking matrix} of $G$ is $B^\top$, the transpose matrix of $B$.
\end{definition}

\begin{remark}
    Although we shall mainly focus on the matrix $B$, we choose to refer to $B^\top$ as the non-backtracking matrix because of historical reasons.
\end{remark}

Now, fix also a directed graph $\mathcal{G}=(\mathcal{V},\mathcal{E})$ on $\mathcal{N}$ nodes and $\mathcal{M}$ edges. 
If $\mathcal{G}$ has an edge from a vertex $v$ to a vertex $w$, we write $v\rightarrow w$ and we denote such an edge by $(v\rightarrow w)$. Given a vertex $v\in\mathcal{V}$, we define its \emph{outdegree} or simply its \emph{degree} as
\begin{equation*}
    \deg_{\mathcal{G}} v\coloneqq\deg v\coloneqq |\{w\in \mathcal{V}:(v\rightarrow w)\in \mathcal{E}\}|.
\end{equation*}We now define the following operators.
\begin{definition}
    The \emph{degree matrix} of $\mathcal{G}$ is the $\mathcal{N}\times \mathcal{N}$ diagonal matrix $\mathcal{D} \coloneqq \mathcal{D}(\mathcal{G}) \coloneqq (\mathcal{D}_{vw})_{v,w\in \mathcal{V}}$ whose diagonal entries are 
\begin{equation*}
    \mathcal{D}_{vv}\coloneqq\deg v.
\end{equation*}
The \emph{adjacency matrix} of $\mathcal{G}$ is the $\mathcal{N}\times \mathcal{N}$ matrix $\mathcal{A}\coloneqq\mathcal{A}(\mathcal{G})\coloneqq(\mathcal{A}_{vw})_{v,w\in \mathcal{V}}$ defined by
\begin{equation*}
    \mathcal{A}_{vw}\coloneqq\begin{cases}1 & \text{if }v\rightarrow w\\
    0 & \text{otherwise}.
    \end{cases}
\end{equation*}If $\mathcal{G}$ has minimum degree $\geq 1$, we also define its \emph{random walk Laplacian} as the $\mathcal{N}\times \mathcal{N}$ matrix
\begin{equation*}
   \mathcal{L}(\mathcal{G})\coloneqq\id-\mathcal{D}^{-1}\mathcal{A},
\end{equation*}where $\id$ denotes the $\mathcal{N} \times \mathcal{N}$ identity matrix.
\end{definition}

We are now ready to define the non-backtracking graph and the non-backtracking Laplacian of a simple graph $G$, as follows.
\begin{definition}
The \emph{non-backtracking graph} of $G$ is the directed graph $\mathcal{NB}(G)$ on vertices $e_1,\ldots,e_{2M}$, that has $B$ as adjacency matrix.\newline
For a simple graph $G$ with minimum degree $\geq 2$, the \emph{non-backtracking Laplacian} of $G$, denoted by $\mathcal{L}\coloneqq\mathcal{L}(G)$, is the random walk Laplacian $\mathcal{L}(\mathcal{NB}(G))$ of $\mathcal{NB}(G)$.
\end{definition}

For simplicity, we shall denote the non-backtracking graph of $G=(V,E)$ by $\mathcal{G}=(\mathcal{V},\mathcal{E})$. Clearly, if $G$ as $N$ nodes and $M$ edges, then $\mathcal{G}$ has $2M$ nodes. Moreover, as shown in \cite{NB-Laplacian}, $\mathcal{G}$ has $\sum_{v\in V}(\deg_{G}v)^2-2M$ edges, and if $G$ has minimum degree $\geq 2$, then the following are equivalent:
\begin{enumerate}
    \item $G$ is not the cycle graph;
    \item $G$ has at least two cycles;
    \item $\mathcal{G}$ is weakly connected;
    \item $\mathcal{G}$ is strongly connected.
\end{enumerate}

\begin{remark}
The \emph{line-digraph} proposed by Harary and Norman \cite{Harary1960} is given by the non-backtracking graph with the addition of all directed edges of the form $[v,w]\rightarrow [w,v]$.
\end{remark}

\begin{remark}
    Assume that $G$ has minimum degree $\geq 2$. The off-diagonal entries of the non-backtracking Laplacian of $G$ are given by
\begin{equation*}
    \mathcal{L}_{ij}=-\frac{B_{ij}}{\deg_{\mathcal{G}}(e_i)}=-\frac{B_{ij}}{\deg_G(\out(e_i))-1}=-\mathbb{P}(e_i\rightarrow e_j),
\end{equation*}for $i\neq j$. Therefore, the non-backtracking Laplacian encodes the probabilities of non-backtracking random walks on $G$. 
\end{remark}
We shall now summarize some of the main results from \cite{NB-Laplacian} on the non-backtracking Laplacian $\mathcal{L}$ associated with the simple graph $G$.
\begin{enumerate}
\item $\mathcal{L}$ has $2M$ eigenvalues (counted with algebraic multiplicity) that sum to $2M$ and are contained in the complex disc $D(1,1)$.
 \item The multiplicity of the eigenvalue $0$ for $\mathcal{L}$ equals the number of connected components of $\mathcal{G}$, and this coincides with the number of connected components of $G$ if none of them are a cycle graph.
 \item $2$ is an eigenvalue if $\mathcal{G}$ is bipartite,  which happens if and only if $G$ is bipartite.
    \item $\mathcal{L}$ is self-adjoint with respect to the \emph{$P$-product} $(\mathbf{x},\mathbf{y})_P \coloneqq\langle \mathbf{x}, P \mathbf{y} \rangle  = \overline{\mathbf{x}}^\top P \mathbf{y}$, where \begin{equation*}
    P\coloneqq\begin{pmatrix}
  \begin{matrix}
  0 & \id \\
\id & 0
  \end{matrix}
\end{pmatrix}
\end{equation*}is a $2M\times 2M$ matrix satisfying $P^\top=P$ and $P^2=\id$. As a consequence,
\item If $(\lambda,\mathbf{x})$ is an eigenpair for $\mathcal{L}$ and $\lambda\in\mathbb{C}$ is not real, then  
    \begin{equation*}
    \sum_{[v, w]}\overline{x_{[v, w]}}\cdot x_{[w, v]}=\sum_{i=1}^M\left(\overline{x_i}\cdot x_{i+M}+\overline{x_{i+M}}\cdot x_{i}\right)=0.
\end{equation*}
\item Let $\sigma(\mathcal{L})$ denote the spectrum of $\mathcal{L}$, seen as a multiset that contains the eigenvalues with their algebraic multiplicity. The spectral gap from $1$ for $\mathcal{L}$, $\varepsilon\coloneqq\min_{\lambda\in \sigma(\mathcal{L})}\left|1-\lambda\right|,$ satisfies the sharp inequality
\begin{equation*}
      \varepsilon \geq \frac{1}{\Delta-1},
 \end{equation*}where $\Delta$ denotes the maximum vertex degree of $G$.
 \item Most cycles of $G$ leave a signature on the spectrum of $\mathcal{L}$ that depends on the length of such cycles and on the degrees of the vertices that they contain. Two examples of this are given by the following results.

 \begin{theorem}[Theorem 5.4 in \cite{NB-Laplacian}]
     Let $d\in\mathbb{N}_{>1}$. If there exists a simple chordless cycle in $G$ whose vertices have constant degree $d$, then $1-\frac{1}{d-1}$ is an eigenvalue for $\mathcal{L}$. If, additionally, such a cycle is even, then also $1+\frac{1}{d-1}$ is an eigenvalue for $\mathcal{L}$.\newline 
Moreover, the geometric multiplicity of $1-\frac{1}{d-1}$ for $\mathcal{L}$ is larger than or equal to the number of $d$--regular simple chordless cycles in $G$, while the geometric multiplicity of $1+\frac{1}{d-1}$ for $\mathcal{L}$ is larger than or equal to the number of $d$--regular even simple chordless cycles in $G$.
 \end{theorem}
 Similarly,
 \begin{theorem}[Theorem 5.5 in \cite{NB-Laplacian}]
     Let $d\in\mathbb{N}_{>2}$. If there exists a simple chordless cycle of length $\ell$ in $G$ such that one vertex has degree $d$ while all other vertices have degree $2$, then $1-\frac{1}{\sqrt[\ell]{d-1}}$ is an eigenvalue for $\mathcal{L}$. If, additionally, such cycle is even, then also $1+\frac{1}{\sqrt[\ell]{d-1}}$ is an eigenvalue for $\mathcal{L}$.\newline 
Moreover, the geometric multiplicity of $1-\frac{1}{\sqrt[\ell]{d-1}}$ for $\mathcal{L}$ is larger than or equal to the number of such cycles in $G$, while the geometric multiplicity of $1+\frac{1}{\sqrt[\ell]{d-1}}$ for $\mathcal{L}$ is larger than or equal to the number of such even cycles in $G$.
 \end{theorem}
 \item 
 \begin{table}[h]
\centering
\begin{tabular}{r|r|rr|rr}
\toprule
$N$ &  \#graphs &      $A$ &    $L$ &  $\mathcal{A}$ &  $\mathcal{L}$ \\
\midrule
$\le$6  &         76 &          0 &       2 &               0 &              0 \\
7  &        510 &         26 &       4 &               0 &              0 \\
8  &      7 459 &        744 &      11 &               2 &              0 \\
9  &    197 867 &     32 713 &     243 &               6 &              0 \\
10 &  9 808 968 &  1 976 884 &  16 114 &          10 130 &            156 \\
\midrule
total & 10 014 880 & 2 010 367 & 16 374 & 10 138 & 156 \\
\bottomrule
\end{tabular}
\caption{(Table 2 in \cite{NB-Laplacian}). Graphs with minimum degree $\ge 2$ not determined by their spectrum, by number of nodes $N$, with respect to the adjacency matrix $A$, the normalized Laplacian $L$, the non-backtracking matrix $\mathcal{A}$ and the non-backtracking Laplacian $\mathcal{L}$.}
\label{tab:md2}
\end{table}
 Two graphs are \emph{cospectral} with respect to a given matrix if they have the same spectrum with respect to that matrix, but they are not isomorphic. In \cite{NB-Laplacian}, computations for graphs with small number of nodes have suggested that the non-backtracking Laplacian has nicer cospectrality properties than the other matrices which are typically considered in spectral graph theory, including the non-backtracking matrix. One example of such computations is given by Table \ref{tab:md2} below, which shows the number of simple graphs with minimum degree $\ge 2$ which are not determined by their spectrum with respect to the adjacency matrix, the normalized Laplacian, the non-backtracking matrix and the non-backtracking Laplacian. 
\end{enumerate}

In summary, the non-backtracking Laplacian is already known to have various notable spectral properties. Since it has been recently defined, it leaves the door open for further theoretical results as well as future applications to network science. Moreover, investigating new properties is made challenging by the fact that it is a non-symmetric matrix. Here in this paper, we make further contributions in the theoretical direction. We prove new results on the non-backtracking graph, and several new results on the non-backtracking Laplacian.\newline 


The rest of this paper is structured as follows. In Sections \ref{section:NBgraph} and \ref{section:efunctions}, we prove some new properties of the non-backtracking graph and the eigenfunctions of $\mathcal{L}$, respectively. In Section \ref{section:partite}, we investigate the following family of graphs. Given $k\geq 1$, a simple graph $G=(V,E)$ is said to be \emph{circularly $k$--partite} if the set of its oriented edges can be partitioned as $\mathcal{V}=\mathcal{V}_1\sqcup \ldots \sqcup \mathcal{V}_k$, where the sets $\mathcal{V}_j$ are non-empty and satisfy the property that
\begin{equation*}
    [v,w]\in \mathcal{V}_j \Longrightarrow [w,z]\in \mathcal{V}_{j+1}, \, \forall z\sim w\,:\, z\neq v,
\end{equation*}where we also let $\mathcal{V}_0:=\mathcal{V}_k$ and $\mathcal{V}_{k+1}:=\mathcal{V}_1$. One can check that all graphs are $1$-circularly partite, and $2$-circularly partite graphs are precisely the bipartite graphs. Therefore, this new family of graphs is a generalization of bipartite graphs. This will be needed in Section \ref{section:gap1}, where we shall prove that, if a graph is circularly $k$--partite, then the spectral gap from $1$ satisfies
\begin{equation*}
    \varepsilon\leq \frac{1}{\sqrt[2M-k]{\prod_{v\in V}(\deg v -1)^{\deg v-1}}},
\end{equation*}and the bound is sharp. Moreover, in Section \ref{section:singular} we study the singular values of the non-backtracking Laplacian in relation to independence numbers, and this is motivated by the long history of inequalities involving independence numbers and eigenvalues in the context of spectral graph theory. Finally, in Section \ref{section:gapR} we prove various bounds for the modulus of the eigenvalues of $\mathcal{L}$, also involving its singular values.

\section{Non-backtracking graph}\label{section:NBgraph}

Our first results relate to the non-backtracking graph. We start by the following. In the introduction, we have mentioned the fact that, if a simple graph $G$ has $N$ nodes and $M$ edges, then  $\mathcal{G}$ has $2M$ nodes and $\sum_{v\in V}(\deg_{G}v)^2-2M$ edges. We now consider the inverse problem:

\begin{lemma}\label{lemma:VertexCount}Let $G$ be a simple graph, and let $\mathcal{G}$ be its non-backtracking graph. If $\mathcal{G}$ has $\mathcal{N}$ nodes, then $G$ has $\mathcal{N}/2$ edges. Moreover, assume that the nodes of $\mathcal{G}$ have degrees $\{d_1-1,\ldots,d_k-1\}$. Then, for each $j=1,\ldots,k$, there exists $c_j\in\mathbb{N}_{\geq 1}$ such that $\mathcal{G}$ has exactly $c_j\cdot d_j$ edges of degree $d_j-1$. Also, $G$ has exactly $c_j$ vertices of degree $d_j$, and the total number of vertices of $G$ is $c_1+\ldots+c_k$.
\end{lemma}
\begin{proof}The fact that $G$ has $\mathcal{N}/2$ edges is straightforward. Now, we know that, for each vertex $v$ in $G$, there are exactly $\deg v$ vertices in $\mathcal{G}$ that have $v$ as an output, and these have degree $\deg v-1$. Hence, if there are exactly $c$ vertices of degree $d$ in $G$, then there are exactly $c\cdot d$ vertices in $\mathcal{G}$ that have degree $d-1$. The claim follows.
\end{proof}

\begin{theorem}
Two simple graphs are isomorphic if and only if their corresponding non-backtracking graphs are isomorphic.
\end{theorem}
\begin{proof}
Fix two simple graphs $G=(V,E)$ and $G'=(V',E')$, and let $\mathcal{G}=(\mathcal{V},\mathcal{E})$ and $\mathcal{G}'=(\mathcal{V}',\mathcal{E}')$ their non-backtracking graphs.\newline If $G$ and $G'$ are isomorphic, then there exists a bijection $f:V\rightarrow V'$ such that 
\begin{equation*}
    v\sim w \text{ in }G \iff  f(v)\sim f(w) \text{ in }G'.
\end{equation*}Let $\varphi:\mathcal{V}\rightarrow \mathcal{V}'$ defined by
\begin{equation*}
    \varphi([v,w]):=[f(v),f(w)].
\end{equation*}Then, $\varphi$ is clearly bijective, and
\begin{align*}
    \omega_1\rightarrow \omega_2 \text{ in }\mathcal{G} \iff& \out(\omega_1)=\inp(\omega_2) \text{ and }\inp(\omega_1)\neq \out(\omega_2)\\ 
    \iff& \inp(\omega_1) \sim \out(\omega_1)=\inp(\omega_2) \sim \out(\omega_2) \text{ in }G \\ &\text{ and }\inp(\omega_1)\neq \out(\omega_2)\\
    \iff& f(\inp(\omega_1)) \sim f(\out(\omega_1))=f(\inp(\omega_2)) \sim f(\out(\omega_2)) \text{ in }G'\\ &\text{ and }f(\inp(\omega_1))\neq f(\out(\omega_2))\\
    \iff& \varphi(\omega_1)\rightarrow \varphi(\omega_2) \text{ in }\mathcal{G}',
\end{align*}implying that $\varphi$ is an isomorphism between $\mathcal{G}$ and $\mathcal{G}'$.\newline
Vice versa, if $\mathcal{G}$ and $\mathcal{G}'$ are isomorphic, then there exists a bijection $\gamma:\mathcal{V}\rightarrow \mathcal{V}'$ such that 
\begin{equation*}
    \omega_1 \rightarrow \omega_2 \text{ in }\mathcal{G} \iff   \gamma(\omega_1) \rightarrow \gamma(\omega_2) \text{ in }\mathcal{G}'.
\end{equation*}Given $v\in V$ of degree $\geq 2$, let $w_1,w_2$ be two distinct neighbors of $v$ in $G$. Then, $[w_1,v]\rightarrow [v,w_2]$ in $\mathcal{G}$, implying that $\gamma([w_1,v])\rightarrow \gamma([v,w_2])$ in $\mathcal{G}'$. Hence,
\begin{equation}\label{eq:gv}
    \out(\gamma([w_1,v]))=\inp(\gamma([v,w_2])).
\end{equation}Since this holds for all distinct neighbors $w_1,w_2$ of $v$, we can define $g(v)\in V'$ to be \eqref{eq:gv}.\newline

Given $v\in V$ of degree $1$ and given $w\sim v$, we define
\begin{equation*}
g(v) := \out(\gamma([w,v])).
\end{equation*}

We claim that $g:V\rightarrow V'$ is an isomorphism between $G$ and $G'$.\newline 
 
Consider $v_1,v_2\in V$ such that $v_1\sim v_2 $ in $G$ and, without loss of generality, assume $\deg(v_1)>1$, as the case in which $\deg(v_1)=\deg(v_2)=1$ is trivial. Observe that
\begin{align*}
  v_1\sim v_2 \text{ in }G &\iff [v_1,v_2]\in\mathcal{G}\\
  &\iff  \gamma([v_1,v_2])\in\mathcal{G}'\\
  &\iff \inp(\gamma([v_1,v_2])) \sim \out(\gamma([v_1,v_2])) \text{ in }G'\\
  &\iff g(v_1)\sim g(v_2) \text{ in } G'.  
\end{align*}Now, if $w_1\in V'$ and $w_1\sim w_2$ in $G'$, then by the surjectivity of $\gamma$ we have that there exist $v_1,v_2\in V$ such that $\gamma([v_1,v_2])=[w_1,w_2]$. Hence, $g(v_2)=w_2$, implying that $g$ is surjective. Moreover, by the pigeonhole principle, since we know that $|V|=|V'|$ as a consequence of Lemma \ref{lemma:VertexCount}, it follows that $g$ is also injective.
\end{proof}



We now ask: What is the fraction of directed graphs that are non-backtracking graphs? We give an upper bound for this quantity in the next theorem, in which isomorphic graphs are not counted only once. 

\begin{theorem}
 Given $\mathcal{N}\in\mathbb{N}$, let $F(\mathcal{N})$ be the fraction of directed graphs on $\mathcal{N}$ nodes that are non-backtracking graphs of (labelled) simple graphs with minimum degree $\geq 1$. Then, $F(\mathcal{N})=0$, if $\mathcal{N}$ is odd, and
\begin{equation*}
F(\mathcal{N})=\frac{\sum_{N\leq 2M}\sum_{k=0}^{N}(-1)^k \binom{N}{k}\binom{\binom{N-k}{2}}{M}}{2^{(2M)(2M-1)}},
\end{equation*}if $\mathcal{N}=2M$ is even. In particular, $F(\mathcal{N})\rightarrow 0$ as $\mathcal{N}\rightarrow\infty$.
\end{theorem}
\begin{proof}
Given $\mathcal{N}\in\mathbb{N}$, there are $\mathcal{N}(\mathcal{N}-1)$ possible directed edges among $\mathcal{N}$ nodes. Hence, the number of directed graphs on $\mathcal{N}$ nodes is $2^{\mathcal{N}(\mathcal{N}-1)}$.\newline 

Similarly, given $N\in\mathbb{N}$, there are $\binom{N}{2}$ possible undirected edges among $N$ nodes. Therefore, given $M\in\mathbb{N}$, the number of (labelled) simple graphs on $N$ nodes and $M$ edges is $\binom{\binom{N}{2}}{M}$.\newline 

Now, for $j=1,\ldots,N$, let $$A_j\coloneqq\{\text{simple graphs on $N$ nodes and $M$ edges where $v_j$ is an isolated vertex}\}.$$
Then, $$\left|A_{j_1}\cap\ldots\cap A_{j_k}\right|=\binom{\binom{N-k}{2}}{M},$$ for any $1\le j_1<\ldots<j_k\le N$. And by the principle of inclusion-exclusion,
\begin{align*}
\left|\cup_{j=1}^NA_j\right|&=\sum_{k=1}^{N}(-1)^{k-1}\sum_{1\le j_1<\ldots<j_k\le N}|A_{j_1}\cap\ldots\cap A_{j_k}|    
\\&= \sum_{k=1}^{N}(-1)^{k-1} \sum_{1\le j_1<\ldots<j_k\le N}\binom{\binom{N-k}{2}}{M}
\\&= \sum_{k=1}^{N}(-1)^{k-1} \binom{N}{k}\binom{\binom{N-k}{2}}{M}.
\end{align*} 
Hence, 
\begin{align*}
 &\#\{\text{simple graphs on $N$ nodes and $M$ edges that have minimum degree }\geq 1\}\\
 &=\#\{\text{simple graphs on $N$ nodes and $M$ edges}\}- |\cup_{j=1}^NA_j|
    \\
&=\binom{\binom{N}{2}}{M}-\sum_{k=1}^{N}(-1)^{k-1} \binom{N}{k}\binom{\binom{N-k}{2}}{M}
\\  
&=\sum_{k=0}^{N}(-1)^k \binom{N}{k}\binom{\binom{N-k}{2}}{M},    
\end{align*}
where we set $\binom{N}{0}=1$, and $\binom{\binom{N-k}{2}}{M}=0$ whenever $\binom{N-k}{2}<M$. Moreover, we can give an upper bound of the above quantity by observing that, for a given $j\in \{1,\ldots,N\}$,
\begin{align*}
     &\#\{\text{simple graphs on $N$ nodes and $M$ edges that have minimum degree }\geq 1\}\\
     &=\sum_{k=0}^{N}(-1)^k \binom{N}{k}\binom{\binom{N-k}{2}}{M}\\   &= \#\{\text{simple graphs on $N$ nodes and $M$ edges}\}- |\cup_{j=1}^NA_j|
    \\    &\leq \#\{\text{simple graphs on $N$ nodes and $M$ edges}\}- |A_1|
    \\
    &=\binom{\binom{N}{2}}{M}-\binom{\binom{N-1}{2}}{M}.
\end{align*}

Now, by Theorem 4 in \cite{cooper2009}, the number of directed graphs on $\mathcal{N}$ nodes that are non-backtracking graphs of simple graphs with minimum degree $\geq 1$ equals the number of simple graphs with $\mathcal{N}/2$ edges and with minimum degree $\geq 1$. This is equal to $0$, if $\mathcal{N}$ is odd, and it is equal to
\begin{equation*}\sum_{N\leq 2M}\sum_{k=0}^{N}(-1)^k \binom{N}{k}\binom{\binom{N-k}{2}}{M},
\end{equation*}if $\mathcal{N}=2M$ is even. Hence, $F(\mathcal{N})=0$, if $\mathcal{N}$ is odd, and
\begin{align*}
F(\mathcal{N})&=\frac{\sum_{N\leq 2M}\sum_{k=0}^{N}(-1)^k \binom{N}{k}\binom{\binom{N-k}{2}}{M}}{2^{(2M)(2M-1)}}\\
&\leq \frac{\sum_{N\leq 2M}\binom{\binom{N}{2}}{M}-\binom{\binom{N-1}{2}}{M}}{2^{(2M)(2M-1)}}\\
&=\frac{\binom{\binom{2M}{2}}{M}}{2^{(2M)(2M-1)}},
\end{align*}if $\mathcal{N}=2M$ is even. Hence, in particular, $F(\mathcal{N})\rightarrow 0$ as $\mathcal{N}\rightarrow\infty$.
\end{proof}

Now, in \cite{NB-Laplacian}, it is shown that a graph is bipartite if and only if its non-backtracking graph is bipartite. We expand this result by proving the following.

\begin{proposition}\label{prop:bipartite}
Let $G=(V,E)$ be a simple graph with minimum degree $\delta\geq 2$, and let $\mathcal{G}=(\mathcal{V},\mathcal{E})$ be its non-backtracking graph. If $\mathcal{G}$ is bipartite, then $[v,w]$ and $[w,v]$ belong to different sets of the bipartition, for each $[v,w]\in \mathcal{V}$.
\end{proposition}

\begin{proof}
We can assume that $G$ is connected, since otherwise we can apply the statement to all connected components of $G$. Moreover, if $G$ is a cycle graph of even length, then $\mathcal{G}$ has multiple components and the claim is trivially satisfied. Therefore, we can also assume that $G$ is not a cycle graph. Since $\delta\geq 2$, there exists a non-backtracking path in $G$ of the form (Figure \ref{fig:bipartite}):
\begin{equation*}
  (v,w),(w=q_1,q_2),\ldots, (q_{r-1},q_r=c_1),\ldots, (c_{s},c_1),
\end{equation*}
for some $r\geq 1$ and for some even $s\geq 4$.

\begin{figure}[h]
    \centering
    \includegraphics[width=7cm]{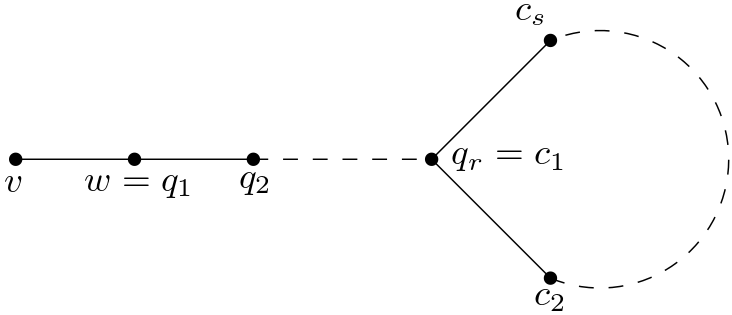}
    \caption{An illustration of the proof of Proposition \ref{prop:bipartite}.}
    \label{fig:bipartite}
\end{figure}

This gives a directed path, in $\mathcal{G}$, of the form
\begin{align*}
    &[v,w]\rightarrow [w,q_2]\rightarrow\ldots\rightarrow [q_{r-1},q_r=c_1]\rightarrow  \ldots\rightarrow  [c_s,c_1]\\
    &\rightarrow[c_1=q_r,q_{r-1}]\rightarrow \ldots \rightarrow[q_2,w]\rightarrow [w,v].
\end{align*}This is a path of length $2(r-1)+s+1$, therefore of odd length, from $[v,w]$ to $[w,v]$. Hence, $[v,w]$ and $[w,v]$ belong to different sets of the bipartition.
\end{proof}

\section{Non-backtracking Laplacian eigenfunctions}\label{section:efunctions}
In this section we investigate some classes of eigenfunctions of the non-backtracking Laplacian. We start by showing a property that all eigenfunctions corresponding to non-zero eigenvalues have to satisfy. We also recall that, as shown in \cite{NB-Laplacian}, $0$ is always an eigenvalue, and the corresponding eigenfunctions are constant on each connected component of $\mathcal{G}$.  

\begin{proposition}\label{prop:0}
Let $G=(V,E)$ be a simple graph with minimum degree $\geq 2$, and let $\mathcal{G}=(\mathcal{V},\mathcal{E})$ be its non-backtracking graph. If $(\lambda, f)$ is an eigenpair for the non-backtracking Laplacian and $\lambda\neq 0$, then
\begin{equation*}
    \sum_{[v,w]\in\mathcal{V}} f([v,w])=0.
\end{equation*}
\end{proposition}
\begin{proof}By \cite{NB-Laplacian}, $\lambda\neq 1$ and, for each $[v,w]\in\mathcal{G}$,
\begin{equation*}
    f([v,w])=\frac{1}{(1-\lambda)}\cdot \frac{1}{(\deg w-1)}\cdot \left( \sum_{[v,w]\rightarrow [w,z]}f([w,z])\right).
\end{equation*}Therefore, for each $w\in V$,
\begin{align*}
   \sum_{v\sim w} f([v,w])&=\frac{1}{(1-\lambda)}\cdot \frac{1}{(\deg w-1)}\cdot \left( \sum_{v\sim w}\,\sum_{z\sim w,\,z\neq v}f([w,z])\right)\\
   &=\frac{1}{(1-\lambda)}\cdot\left( \sum_{z\sim w}f([w,z])\right),
\end{align*}implying that
\begin{equation*}
    \sum_{w\in V}\sum_{v\sim w} f([v,w])=\frac{1}{(1-\lambda)}\cdot\left(\sum_{w\in V}\sum_{z\sim w}f([w,z])\right),
\end{equation*}i.e.,
\begin{equation*}
    \sum_{[v,w]\in\mathcal{V}} f([v,w])=\frac{1}{(1-\lambda)}\cdot\left( \sum_{[v,w]\in\mathcal{V}} f([v,w])\right).
\end{equation*}Since $\lambda\neq 0$, this implies that
\begin{equation*}
    \sum_{[v,w]\in\mathcal{V}} f([v,w])=0.
\end{equation*}
\end{proof}

\begin{remark}
As a consequence of Proposition \ref{prop:0}, if $f$ is a real eigenfunction for $\mathcal{L}$ corresponding to a non-zero eigenvalue, then $f$ must attain both positive and negative values.
\end{remark}

We now introduce and study symmetric and antisymmetric functions on $\mathcal{V}$.

\begin{definition}
    A function $f:\mathcal{V}\to\mathbb{C}$ is \emph{symmetric} (respectively, \emph{antisymmetric}) if 
    \begin{equation*}
        f([v,w])=f([w,v])
    \end{equation*} (respectively, $f([v,w])=-f([w,v])$), for all $[v,w]\in\mathcal{V}$.
\end{definition}

\begin{proposition}\label{prop:-f}
Let $G=(V,E)$ be a simple graph with minimum degree $\geq 2$, and let $\mathcal{G}=(\mathcal{V},\mathcal{E})$ be its non-backtracking graph. If $(\lambda, f)$ is an eigenpair for the non-backtracking Laplacian and $f$ is antisymmetric, then
\begin{equation}\label{eq:flow}
   \frac{1}{(\deg w-1)}\cdot \left( \sum_{[v,w]\rightarrow [w,z]}f([w,z])\right)= \frac{1}{(\deg v-1)}\cdot \left( \sum_{[y,v]\rightarrow [v,w]}f([y,v])\right),
\end{equation}for each $[v,w]\in\mathcal{V}$, and $\lambda$ is real.
\end{proposition}
\begin{proof}
By \cite{NB-Laplacian}, $\lambda\neq 1$ and, for each $[v,w]\in\mathcal{G}$,
\begin{equation*}
    f([v,w])=\frac{1}{(1-\lambda)}\cdot \frac{1}{(\deg w-1)}\cdot \left( \sum_{[v,w]\rightarrow [w,z]}f([w,z])\right).
\end{equation*}Hence, if $f([v,w])=-f([w,v])$ for each $[v,w]\in\mathcal{G}$, then
\begin{align*}
     f([v,w])&=\frac{1}{(1-\lambda)}\cdot \frac{1}{(\deg w-1)}\cdot \left( \sum_{[v,w]\rightarrow [w,z]}f([w,z])\right)\\
     &=-f([w,v])\\
     &=-\frac{1}{(1-\lambda)}\cdot \frac{1}{(\deg v-1)}\cdot \left( \sum_{[w,v]\rightarrow [v,y]}f([v,y])\right)\\
     &=\frac{1}{(1-\lambda)}\cdot \frac{1}{(\deg v-1)}\cdot \left( \sum_{[y,v]\rightarrow [v,w]}f([y,v])\right).
\end{align*}Therefore, for all $[v,w]\in\mathcal{V}$,
\begin{equation*}
   \frac{1}{(\deg w-1)}\cdot \left( \sum_{[v,w]\rightarrow [w,z]}f([w,z])\right)= \frac{1}{(\deg v-1)}\cdot \left( \sum_{[y,v]\rightarrow [v,w]}f([y,v])\right).
\end{equation*}This proves the first claim. Now, assume by contradiction that $\lambda\neq \overline{\lambda}$. Then, by Theorem 3.9 in \cite{NB-Laplacian} and by the condition that $f$ has to satisfy,

\begin{equation*}
    0= \sum_{[v, w]} f([v, w])\cdot \overline{f([w, v])}=-\sum_{[v, w]} f([v, w])\cdot \overline{f([v, w])}.
\end{equation*}Since $f([v, w])\cdot \overline{f([v, w])}\geq 0$, this implies that $f([v, w])=0$ for each $[v,w]\in \mathcal{V}$, which is a contradiction.
\end{proof}

\begin{remark}
The condition in Proposition \ref{prop:-f} can be interpreted as follows. If we see directed edges in terms of flows, then \eqref{eq:flow} says that, for each $[v,w]\in\mathcal{G}$, the average of what flows out equals the average of what flows in.
\end{remark}
\begin{remark}
If $G=(V,E)$ is $d$-regular, then the eigenfunctions of $\mathcal{L}$ coincide with the ones of $B$ \cite{NB-Laplacian}. In this case, it is known that there are exactly $|E|-|V|+1$ eigenfunctions satisfying Proposition \ref{prop:-f} \cite{torres-1}. They coincide with the eigenfunctions of $1$ for $B$, and they are eigenfunctions of $1-\frac{1}{d-1}$ for $\mathcal{L}$.  
\end{remark}

Analogously to Proposition \ref{prop:-f}, one can prove the following.

\begin{proposition}\label{prop:+f}
Let $G=(V,E)$ be a simple graph with minimum degree $\geq 2$, and let $\mathcal{G}=(\mathcal{V},\mathcal{E})$ be its non-backtracking graph. If $(\lambda, f)$ is an eigenpair for the non-backtracking Laplacian and $f$ is symmetric, then
\begin{equation*}
   \frac{1}{(\deg w-1)}\cdot \left( \sum_{[v,w]\rightarrow [w,z]}f([w,z])\right)= \frac{1}{(\deg v-1)}\cdot \left( \sum_{[y,v]\rightarrow [v,w]}f([y,v])\right),
\end{equation*}for each $[v,w]\in\mathcal{V}$, and $\lambda$ is real.
\end{proposition}

\begin{remark}
If $G=(V,E)$ is $d$-regular, then the eigenfunctions that satisfy Proposition \ref{prop:+f} are exactly the eigenfunctions of $-1$ for $B$, and the eigenfunctions of $1+\frac{1}{d-1}$ for $\mathcal{L}$. They are $|E|-|V|+1$ if $G$ is bipartite, and they are $|E|-|V|$ when it is not \cite{torres-1}.
\end{remark}

\begin{definition}
Let $G=(V,E)$ be a simple graph. The \emph{line graph}  $LG$ of $G$ has $E$ as vertex set, and it is such that $e_i\sim e_j$ in $LG$ if and only if $e_i$ and $e_j$ share a common vertex in $G$.
\end{definition}

Now, recall that, in the Introduction, we defined the $2M\times 2M$ matrix \begin{equation*}
    P=\begin{pmatrix}
  \begin{matrix}
  0 & \id \\
\id & 0
  \end{matrix}
\end{pmatrix},
\end{equation*}
which satisfies $P^\top=P$ and $P^2=\id$.

\begin{lemma}
    The matrix $\mathcal{L}P$ is symmetric.
\end{lemma}
\begin{proof}
    As shown in \cite[Theorem 3.9]{NB-Laplacian}, $\mathcal{L}^\top=P\mathcal{L}P$. Together with the fact that $P^{-1}=P$, this implies that $P\mathcal{L}^\top=\mathcal{L}P$, and using $P^\top=P$ this implies that $(\mathcal{L}P)^\top=\mathcal{L}P$.
\end{proof}

In the next proposition we show that, in order to study symmetric eigenfunctions for $\mathcal{L}$, we can reduce to study eigenfunctions of the symmetric real matrix $\mathcal{L}P$, 
or to study the spectrum for the random walk Laplacian of $LG$.

\begin{proposition}Let $G=(V,E)$ be a simple graph with minimum degree $\ge2$, and let $\mathcal{G}=(\mathcal{V},\mathcal{E})$ be its non-backtracking graph. 
If $f:\mathcal{V}\to\mathbb{C}$ is symmetric, then $(\lambda,f)$ is an eigenpair for $\mathcal{L}$ if and only if $(\lambda,f)$ is an eigenpair for $\mathcal{L}P$.\newline Moreover, in this case, let $\tilde{f}:E\rightarrow \mathbb{R}$ be defined by
$\tilde{f}(\{v,w\}):=f([v,w]).$ Then, $(\lambda,\tilde{f})$ is an eigenpair for the random walk Laplacian of $LG$.
\end{proposition}

\begin{proof}
Note that $f$ is symmetric if and only if $Pf=f$. Thus, if $(\lambda,f)$ is an eigenpair of $\mathcal{L}$, we have  that $\mathcal{L}Pf=\mathcal{L}f=\lambda f$, which implies that  $(\lambda,f)$ is an eigenpair for $\mathcal{L}P$. Conversely, if  $(\lambda,f)$ is an eigenpair of $\mathcal{L}P$, then $\mathcal{L}f=\mathcal{L}PPf=\mathcal{L}Pf=\lambda f$, and hence $(\lambda,f)$ is an eigenpair for $\mathcal{L}$. This proves the first claim.\newline 

Let now $\tilde{f}:E\rightarrow \mathbb{R}$ be defined by
$\tilde{f}(\{v,w\}):=f([v,w])=f([w,v])$, and let $\tilde{L}$ denote the random walk Laplacian of the line graph $LG$. \newline
By definition of $\tilde{L}$, $\tilde{f}$ and $\mathcal{L}$, by the assumptions and by Proposition \ref{prop:+f}, we have 
\begin{align*}
&\tilde{L} \tilde{f}(\{v,w\})\\&=\tilde{f}(\{v,w\})-\frac{1}{\deg_{LG}(\{v,w\})}\cdot \left(\sum_{\{u,u'\}\in E:\{u,u'\}\cap \{v,w\}\ne\emptyset}f(\{u,u'\})\right)
\\&=\frac{f([v,w])+f([w,v])}{2}-\frac{1}{\deg_{\mathcal{G}}([v,w])+\deg_{\mathcal{G}}([w,v])}\cdot \left(\sum_{[v,w]\to[w,z]}f([w,z])+\sum_{[y,v]\to[v,w]}f([y,v])\right)
\\&=\frac{f([v,w])+f([w,v])}{2}-\frac{\sum\limits_{[v,w]\to[w,z]}f([w,z])+\sum\limits_{[y,v]\to[v,w]}f([y,v])}{\deg_G w-1+\deg_G v-1}
\\&=f([v,w])-\frac{1}{\deg_G w-1}\cdot\left(\sum_{[v,w]\to[w,z]}f([w,z])\right)
\\&=\mathcal{L}f([v,w])
\\&=\lambda f([v,w])\\ &=\lambda \tilde{f}(\{v,w\}).
\end{align*}
Therefore, $(\lambda,\tilde{f})$ is an eigenpair for $\Tilde{L}$.
\end{proof}

\begin{remark}
        Let $G=(V,E)$ be a simple graph with minimum degree $\geq 2$, and let $\mathcal{G}=(\mathcal{V},\mathcal{E})$ be its non-backtracking graph. Let $[v_0,v_1]\in \mathcal
    V$ and let
    $$\mathcal{N}_k([v_0,v_1])\coloneqq \{[v_k,v_{k-1}]\,:\,[v_0,v_1]\rightarrow[v_1,v_2]\rightarrow\ldots\rightarrow[v_k,v_{k+1}]\},$$
   for $k\geq 1$. If $(\lambda, f)$ is an eigenpair for the non-backtracking Laplacian, then $\lambda\neq 1$ by \cite{NB-Laplacian}, and
\begin{align*}
    f([v_0,v_1])&=\frac{1}{(1-\lambda)}\cdot \frac{1}{(\deg v_1-1)}\cdot \left( \sum_{[v_0,v_1]\rightarrow [v_1,v_2]}f([v_1,v_2])\right)\\
    &=\frac{1}{(1-\lambda)^{k}}\cdot \left(\sum_{[v_0,v_1]\rightarrow\ldots\rightarrow[v_k,v_{k+1}]}\frac{1}{(\deg v_1-1)}\cdots\frac{1}{(\deg v_k-1)}\cdot f([v_k,v_{k+1}])\right).
\end{align*}Therefore, the value of $f$ on $[v_0,v_1]$ is completely determined by the value of $f$ on $\mathcal{N}_k([v_0,v_1])$, for any given $k\geq 1$.
\end{remark}

\section{Circularly partite graphs}\label{section:partite}

In this section, we generalize the notion of bipartite graphs by defining \emph{circularly $k$--partite graphs}. We also generalize the fact that $0\in\sigma(\mathcal{L})$ for all graphs and $2\in\sigma(\mathcal{L})$ for all bipartite graphs, by proving that, for all circularly $k$--partite graphs, if $\omega\in\mathbb{C}$ is a $k$-th root of unit, then $1-\omega\in\sigma(\mathcal{L})$.

\begin{definition}
Let $k\in\mathbb{N}_{\geq 1}$. A simple graph $G=(V,E)$ is \emph{circularly $k$--partite} if the set of its oriented edges can be partitioned as $\mathcal{V}=\mathcal{V}_1\sqcup \ldots \sqcup \mathcal{V}_k$, where the sets $\mathcal{V}_j$ are non-empty and satisfy the property that
\begin{equation*}
    [v,w]\in \mathcal{V}_j \Longrightarrow [w,z]\in \mathcal{V}_{j+1}, \, \forall z\sim w\,:\, z\neq v,
\end{equation*}where we also let $\mathcal{V}_0:=\mathcal{V}_k$ and $\mathcal{V}_{k+1}:=\mathcal{V}_1$. 
\end{definition}

Note that the above definition of circularly $k$--partite graphs is based on the set of oriented edges, but it does not require to construct the non-backtracking graph.\newline 

It is easy to check that all graphs are circularly $1$--partite. Moreover, the path graph of length $M$ is circularly $k$--partite for each $k\in\{1,\ldots,2M\}$, and the cycle graph on $N$ nodes is circularly $k$--partite if and only if $N$ is a multiple of $k$. Also, $G$ is circularly $2$--partite if and only if its non-backtracking graph is bipartite, therefore if and only if $G$ itself is bipartite. More generally, for any $k\geq 2$, it is easy to see that if $G$ is circularly $k$--partite, then its non-backtracking graph is a $k$--partite graph. However, the inverse implication does not always hold, as shown by the next example.

\begin{figure}[h]
    \centering
    \includegraphics[width=9cm]{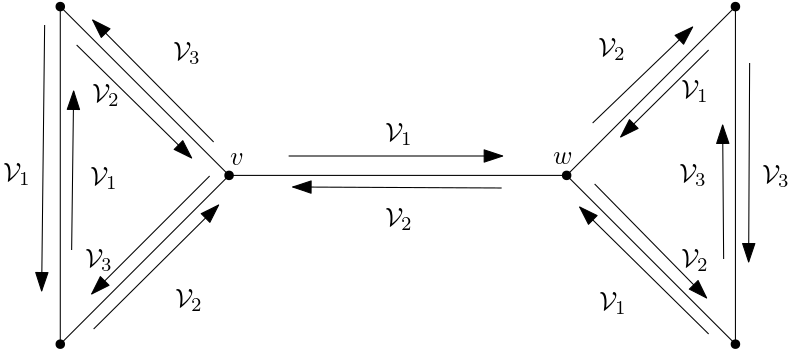}
    \caption{The graph from Example \ref{counterex} is not circularly $3$--partite.}
    \label{fig:counterex}
\end{figure}

\begin{example}\label{counterex}
Consider the simple graph $G$ in Figure \ref{fig:counterex}. To check whether $G$ is circularly $3$--partite, we can start by assuming, without loss of generality, that $[v,w]\in \mathcal{V}_1$. By definition of circularly $3$-partite, we can then inductively indicate to which set $\mathcal{V}_i$ each oriented edge should belong to. However, as shown in the figure, we would then have two oriented edges that belong to $\mathcal{V}_2$ and that go towards an oriented edge in $\mathcal{V}_1$, through the node $v$. This implies that $G$ is not circularly $3$-partite. However, the construction that we obtain shows that the non-backtracking graph of $G$ is $3$-partite.
\end{example}

\begin{figure}[h]
    \centering
    \includegraphics[width=9cm]{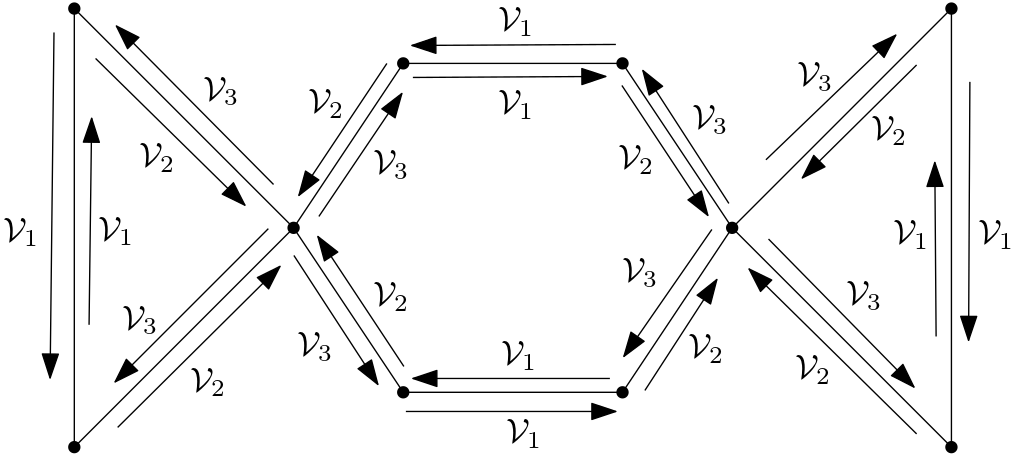}
    \caption{A circularly $3$--partite graph.}
    \label{fig:3partite}
\end{figure}

In Figure \ref{fig:3partite} we give an example of a circularly $3$--partite graph. Another example of circularly partite graphs are the petal graphs:

\begin{definition}The \emph{petal graph} with $p\geq 1$ petals of length $k\geq 3$ is the graph given by $p$ cycle graphs of length $k$, all having a common central vertex.
\end{definition}

Clearly, if $G$ is a petal graph whose petals have length $k$, then $G$ is circularly $k$--partite.

\begin{theorem}\label{thm:partite}
Let $k\in\mathbb{N}_{\geq 1}$. Let $G=(V,E)$ be a simple graph with minimum degree $\geq 2$, and let $\mathcal{G}=(\mathcal{V},\mathcal{E})$ be its non-backtracking graph. If $G$ is circularly $k$--partite and $\omega\in\mathbb{C}$ is a $k$-th root of unit, then $\omega\in\sigma(\mathcal{D}^{-1}\mathcal{A})$, or equivalently, $1-\omega\in\sigma(\mathcal{L})$.
\end{theorem}

\begin{proof}
Let $f:\mathcal{V}\rightarrow\mathbb{C}$ be defined by
\begin{equation*}
    f([v,w]):=\omega^j \iff [v,w]\in \mathcal{V}_j.
\end{equation*}Then, given $[v,w]\in \mathcal{V}_j$,
\begin{equation*}
    \omega \cdot f([v,w])=\omega^{j+1}=\frac{1}{\deg w-1}\cdot \left(\sum_{\substack{z\sim w\,:\\ z\neq v}} f([w,z])\right),
\end{equation*}since $[w,z]\in \mathcal{V}_{j+1}$ for all $z\sim w$ such that $z\neq v$. Hence, $(\omega,f)$ is an eigenpair for $\mathcal{D}^{-1}\mathcal{A}$. Equivalently, $(1-\omega,f)$ is an eigenpair for $\mathcal{L}$.
\end{proof}

\begin{remark}
Theorem \ref{thm:partite} applied to $k=1$ simply says that $0\in \sigma(\mathcal{L})$ for all graphs, while for $k=2$ it says that $2\in \sigma(\mathcal{L})$ for all bipartite graphs. If we apply it to petal graphs whose petals have size $k$, we can infer that the $k$-th roots of unit are eigenvalues for $\mathcal{D}^{-1}\mathcal{A}$ in this case.
\end{remark}

\section{Spectral gap from 1}\label{section:gap1}
Given an operator $\mathcal{O}$, we let $\sigma(\mathcal{O})$ denote its spectrum, seen as a multiset that contains the eigenvalues with their algebraic multiplicity, and we let
\begin{equation*}
    \rho(\mathcal{O}):=\max_{\lambda\in\sigma(\mathcal{O})} |\lambda|, \quad r(\mathcal{O}):=\min_{\lambda\in\sigma(\mathcal{O})} |\lambda|.
\end{equation*}

We recall that, for a simple graph with minimum degree $\geq 2$, $1$ is not in the spectrum of the non-backtracking Laplacian $\mathcal{L}$, hence the spectral gap from $1$, 
\begin{equation*}
\varepsilon\coloneqq\min_{\lambda\in \sigma(\mathcal{L})}\left|1-\lambda\right|=\min_{\lambda\in \sigma(\mathcal{D}^{-1}\mathcal{A})}|\lambda|=r(\mathcal{D}^{-1}\mathcal{A}),
\end{equation*}is positive. In \cite{NB-Laplacian} it is shown that, for a graph with maximum vertex degree $\Delta$, $$\varepsilon \geq \frac{1}{\Delta-1},$$ and the inequality is sharp. Here we prove a sharp upper bound for $\varepsilon$, as well as a lower bound for 
\begin{equation*}
    \mathcal{E}:=\max_{\lambda\in \sigma(\mathcal{L})\setminus\{0\}}\left|1-\lambda\right|=\max_{\lambda\in \sigma(\mathcal{D}^{-1}\mathcal{A})\setminus\{1\}}|\lambda|,
\end{equation*}where the above multiset difference $\sigma(\mathcal{D}^{-1}\mathcal{A})\setminus\{0\}$ means that we are removing one instance of $0$ from $\sigma(\mathcal{D}^{-1}\mathcal{A})$.

\begin{theorem}\label{thm:epsilon}
Let $k\in\mathbb{N}_{\geq 1}$. If $G$ is a circularly $k$--partite graph, then
\begin{equation*}
    \varepsilon\leq \frac{1}{\sqrt[2M-k]{\prod_{v\in V}(\deg v -1)^{\deg v-1}}}\leq\mathcal{E}.
\end{equation*}If $k>1$, then $\mathcal{E}=1$.
\end{theorem}
\begin{proof}
We have \begin{align*}
    \prod_{\lambda\in \sigma(\mathcal{D}^{-1}\mathcal{A})\,:\,|\lambda|\neq  1}\left|\lambda\right|&=\left|\prod_{\lambda\in \sigma(\mathcal{D}^{-1}\mathcal{A})}\lambda\right|
    \\&=\left|\det(\mathcal{D}^{-1}\mathcal{A})\right|\\ &=\left|\frac{\det(P)\det(\mathcal{D}^{-1})\det(\mathcal{A})}{\det(P)}\right|\\
    &=\left|\frac{\det(P\mathcal{A})}{\det (\mathcal{D})}\right|\\ &=\frac{\prod_{v\in V}(\deg v-1)}{\prod_{v\in V}(\deg v -1)^{\deg v}}\\
    &=\frac{1}{\prod_{v\in V}(\deg v -1)^{\deg v-1}},
\end{align*}where we have used the following facts:
\begin{itemize}
    \item The spectrum of $P$ is given by $1$, with multiplicity $M$, and $-1$, with multiplicity $M$.
    \item The spectrum of $\mathcal{D}$ is given by its diagonal entries. These are $\deg v-1$ for each element of the form $[w,v]\in \mathcal{V}$, and there are $\deg v$ of these elements for each $v\in V$.
    \item The spectrum of $P\mathcal{A}$ is equal to the spectrum of $\mathcal{A}P$. By Lemma 4.3 in \cite{NB-Laplacian}, this is given by $-1$, with multiplicity $2M - N$, and $\deg v-1$, for $v\in V$.
\end{itemize}
Now, by Theorem \ref{thm:partite}, there are at least $k$ eigenvalues $\mu_1,\ldots,\mu_k$ of $\mathcal{D}^{-1}\mathcal{A}$ with modulus $1$. Since $\varepsilon\leq |\lambda|\leq \mathcal{E}$ for all $\lambda\in \sigma(\mathcal{D}^{-1}\mathcal{A})\setminus\{\mu_1,\ldots,\mu_k\}$, this implies that
\begin{equation*}
    \varepsilon^{2M-k}\leq \frac{1}{\prod_{v\in V}(\deg v -1)^{\deg v-1}}\leq \mathcal{E}^{2M-k},
\end{equation*}therefore
\begin{equation*}
    \varepsilon\leq \frac{1}{\sqrt[2M-k]{\prod_{v\in V}(\deg v -1)^{\deg v-1}}}\leq \mathcal{E}.
\end{equation*}If, in particular, $k>1$, then there are at least two eigenvalues of $\mathcal{D}^{-1}\mathcal{A}$ with modulus $1$, implying that $\mathcal{E}=1$ in this case.
\end{proof}

 \begin{corollary}For any graph $G$ with minimum degree $\delta\geq 2$ and maximum degree $\Delta$,
 \begin{equation*}
     \varepsilon\leq \frac{1}{\sqrt[\Delta]{(\delta-1)^{(\delta-1)}}}.
 \end{equation*}In particular, $\varepsilon\rightarrow 0$ as $\delta\rightarrow\infty$, if $\Delta=O(\delta)$.
 \end{corollary}
 \begin{proof}
  It follows from the fact that $$2M=\sum_{v\in V}\deg v\leq N\Delta,$$ together with Theorem \ref{thm:epsilon}. 
 \end{proof}
 
We now make the following
\begin{conjecture}\label{conj:epsilon}
If $G$ is not the cycle graph, then\begin{equation*}
    \varepsilon\leq \frac{1}{\delta-1}.
\end{equation*}
\end{conjecture}
By \cite{NB-Laplacian}, we know that Conjecture \ref{conj:epsilon} holds for all regular graphs, for which $\varepsilon= \frac{1}{\delta-1}$, as well as for all graphs having at least one cycle that satisfies Theorem 5.6 in \cite{NB-Laplacian} (as for instance a cycle in which all vertices have the same degree, or a cycle in which exactly one vertex has degree larger than $2$). However, as shown by the next examples, the bound in the above conjecture is not always better than the bound in Theorem \ref{thm:epsilon}. 

\begin{example}Let $G$ be the wheel graph on $N\geq 4$ nodes. Then, $\delta=3$, $\Delta=N-1$ and $M=2\Delta$. Also, Conjecture \ref{conj:epsilon} in this case tells us that
\begin{equation*}
   \varepsilon\leq  \frac{1}{\delta-1}=\frac{1}{2},
\end{equation*}which is always true. In fact, computationally we observed that $\varepsilon=\frac{1}{2}$ for $N=4$, and then $\varepsilon$ decreases for larger $N$. From Theorem \ref{thm:epsilon}, on the other hand, we can infer that
\begin{equation*}
    \varepsilon \leq \frac{1}{\sqrt[2M-1]{\prod_{v\in V}(\deg v -1)^{\deg v-1}}}= \frac{1}{\sqrt[4\Delta-1]{(\Delta-1)^{(\Delta-1)}\cdot 2^{2\Delta}}}=:g(\Delta),
\end{equation*}where for instance
\begin{equation*}
    g(\Delta)\approx \begin{cases}0.41 & \text{for }\Delta=10,\\
    0.35 & \text{for }\Delta=20,\\
    0.22 & \text{for }\Delta=100.
    \end{cases}
\end{equation*}Hence, the bound in Theorem \ref{thm:epsilon} is better than the bound in Conjecture \ref{conj:epsilon} in these cases.
\end{example}

The next theorem gives us all the eigenvalues of the non-backtracking Laplacian for the petal graph. It will allow us, in Corollary \ref{corollary:sharp} below, to show that the first bound in Theorem \ref{thm:epsilon} is sharp.

\begin{theorem}\label{thm:petal}
Let $N\geq 5$, $p\geq 2$ and $k\geq 3$. Let $G=(V,E)$ be the petal graph on $N$ nodes, consisting of $p$ petals of length $k$, and let $\mathcal{G}=(\mathcal{V},\mathcal{E})$ be its non-backtracking graph. Let also $\omega_1,\ldots,\omega_k$ be the $k$-th roots of unit. Then, the eigenvalues of $\mathcal{D}^{-1}\mathcal{A}$ are
\begin{equation*}
    \omega_1,\ldots,\omega_k,\frac{\omega_1}{\sqrt[k]{2p-1}},\ldots,\frac{\omega_k}{\sqrt[k]{2p-1}}.
\end{equation*}Equivalently, the eigenvalues of $\mathcal{L}$ are
\begin{equation*}
    1-\omega_1,\ldots,1-\omega_k,1-\frac{\omega_1}{\sqrt[k]{2p-1}},\ldots,1-\frac{\omega_k}{\sqrt[k]{2p-1}}.
\end{equation*}
\end{theorem}
\begin{proof}
By Theorem \ref{thm:partite}, $\omega_1,\ldots,\omega_k$ are in the spectrum of $\mathcal{D}^{-1}\mathcal{A}$. Now, let $x\in V$ be the central node of $G$, and let $C=\{x,v_2,\ldots,v_k\}$ be the vertex set of one of the petals. Let also $v_1:=v_{k+1}:=x$, let
\begin{equation*}
    \lambda:=\frac{\omega_i}{\sqrt[k]{2p-1}},
\end{equation*}for some $i\in[k]$, and let $f:\mathcal{V}\rightarrow \mathbb{C}$ be defined by
\begin{equation*}
    f([v,w]):=\begin{cases}
    \lambda^j, &\text{ if }[v,w]=[v_j,v_{j+1}],\\
     -\lambda^{k+1-j}, &\text{ if }[v,w]=[v_{j-1},v_j],\\
     0, &\text{ otherwise.}
    \end{cases}
\end{equation*}Then, $f$ is an eigenfunction for $\lambda$. This proves that
\begin{equation*}
    \omega_1,\ldots,\omega_k,\frac{\omega_1}{\sqrt[k]{2p-1}},\ldots,\frac{\omega_k}{\sqrt[k]{2p-1}}
\end{equation*}are eigenvalues of $\mathcal{D}^{-1}\mathcal{A}$. We want to show that there are no other eigenvalues.\newline 

Let $(\mu,g)$ be an eigenpair of $\mathcal{D}^{-1}\mathcal{A}$, and let again $C=\{x,v_2,\ldots,v_k\}$ be the vertex set of one of the petals. Observe that
\begin{align*}
    &g([v_2,v_3])=\mu\cdot g([x,v_2]),\\ &\ldots\\ &g([v_k,v_{k+1}])=\mu^{k-1}\cdot g([x,v_2])
\end{align*}and
\begin{align*}
&g([v_k,v_{k-1}])=\mu\cdot g([x,v_k]),\\ &\ldots\\ &g([v_2,v_1])=\mu^{k-1}\cdot g([x,v_k]).
\end{align*}Therefore, it is enough to know the values of $g$ on the elements that have input $x$, in order to know all its values. Now, let $e_1,\ldots,e_{2p}\in\mathcal{V}$ be the elements of $\mathcal{V}$ whose input is $x$, and write $e_i=e_j^\top$ if $i\neq j$ and $e_i$, $e_j$ belong to the same petal. By Proposition \ref{prop:0} and by the above observations,
\begin{align*}
   0&=\sum_{[v,w]\in\mathcal{V}} g([v,w])=\sum_{i=1}^{2p}\left(g(e_i)+\mu \cdot g(e_i)+\ldots+\mu^{k-1} \cdot g(e_i)\right)\\
   &=\left(1+\mu+\ldots+\mu^{k-1}\right)\cdot\left(\sum_{i=1}^{2p} g(e_i)\right).
\end{align*}If $\mu\notin \{\omega_1,\ldots,\omega_k\}$, then $1+\mu+\ldots+\mu^{k-1}\neq 0$, implying that $\sum_{i=1}^{2p} g(e_i)=0$. Therefore, given $j\in [2p]$,
\begin{equation*}
    g(e_j)=\sum_{i\neq j} g(e_i).
\end{equation*}Now, since $(\mu,g)$ is an eigenpair and $\deg x=2p-1$, we have that
\begin{equation*}
   \mu\cdot g(e_j^{-1})=\mu\cdot \mu^{k-1}\cdot g(e_j^\top)=\frac{1}{2p-1}\cdot \left(\sum_{i\neq j}g(e_i)\right)=\frac{1}{2p-1}\cdot g(e_j).
\end{equation*}Hence,
\begin{equation*}
    \mu^k\cdot g(e_j^\top)=\frac{1}{2p-1}\cdot g(e_j)
\end{equation*}and similarly
\begin{equation*}
   g(e_j)=\frac{1}{ \mu^k}\cdot \frac{1}{2p-1}\cdot g(e_j^\top),
\end{equation*}implying that
\begin{equation*}
    \mu^k\cdot g(e_j^\top)=\frac{1}{2p-1}\cdot \frac{1}{ \mu^k}\cdot \frac{1}{2p-1}\cdot g(e_j^\top).
\end{equation*}For $g(e_j^\top)\neq 0$, which always exists, this implies that $|\mu^{k}|=\left|\frac{1}{2p-1}\right|$, and therefore
\begin{equation*}
    \mu\in\left\{\frac{\omega_1}{\sqrt[k]{2p-1}},\ldots,\frac{\omega_k}{\sqrt[k]{2p-1}}\right\}.
\end{equation*}This proves the claim.
\end{proof}

\begin{corollary}\label{corollary:sharp}
The first bound in Theorem \ref{thm:epsilon} is sharp.
\end{corollary}

\begin{proof}
Let $N\geq 5$, $p\geq 2$ and $k\geq 3$. Let $G=(V,E)$ be the petal graph on $N$ nodes, consisting of $p$ petals of length $k$. Then, $\delta=2$, $\Delta=2p$ and $2M=pk$. In this case, Theorem \ref{thm:epsilon} tells us that
\begin{equation*}
    \varepsilon \leq \frac{1}{\sqrt[2M-k]{\prod_{v\in V}(\deg v -1)^{\deg v-1}}}=\frac{1}{\sqrt[2pk-k]{(2p -1)^{2p-1}}}=\frac{1}{\sqrt[k]{(2p -1)}},
\end{equation*}and by Theorem \ref{thm:petal} we know that equality holds.
\end{proof}

\section{Singular values and independence numbers}\label{section:singular}

The positive semidefinite matrix 
$\mathcal{L}^\top\mathcal{L}$ has $2M$ real, non-negative eigenvalues that we denote by
$$0=\lambda_1(\mathcal{L}^\top\mathcal{L})\le \lambda_2(\mathcal{L}^\top\mathcal{L})\le\ldots\le \lambda_{2M}(\mathcal{L}^\top\mathcal{L}).$$
Their square roots are the \emph{singular values} of $\mathcal{L}$, and we denote them by 
$$0=s_1(\mathcal{L})\le s_2(\mathcal{L})\le\ldots\le s_{2M}(\mathcal{L}).$$
In this section, we study the singular values of $\mathcal{L}$ in relation to independence numbers of $\mathcal{G}$, that we introduce here. Such numbers are inspired by the classical definition of independence number for simple graphs, and so are the results in this section. Notably, the classical spectral bound for the independence number, namely, the inertia bound, can be used to give a spectrum-based proof for the famous Erdős-Ko-Rado Theorem \cite{Godsil16}, and its extended version to signed graphs actually gives an excellent proof for the Sensitivity Conjecture \cite{Huang19}. This 
relation between eigenvalues and  independence numbers is very important in spectral graph theory. 

\begin{definition}
    Two distinct vertices $[v_1,w_1]$ and $[v_2,w_2]$ in $\mathcal{V}$ are \emph{in-adjacent} (respectively, \emph{out-adjacent}) if they have the same output (respectively, the same input), i.e., $w_1=w_2$ (respectively, $v_1=v_2$). They are \emph{adjacent} if either $[v_1,w_1]\rightarrow [v_2,w_2]$ or $[v_1,w_1]\rightarrow [v_2,w_2]$ in $\mathcal{E}$.\newline
A set $S\subset \mathcal{V}$ is \emph{out-independent} if any two vertices in $S$ are not out-adjacent. The \emph{out-independence number} of $G$, denoted $\alpha_{out}(\mathcal{G})$, is the cardinality of the largest independent set. \newline
Similarly, a set $S\subset \mathcal{V}$ is \emph{strong out-independent} if any two vertices in $S$ are neither out-adjacent, nor adjacent. The \emph{strong out-independence number} of $G$, denoted $\alpha_{s\text{-}out}(\mathcal{G})$, is the cardinality of the largest strong out-independent set. 
\end{definition}

    Clearly, $\alpha_{s\text{-}out}(\mathcal{G})\le \alpha_{out}(\mathcal{G})$.\newline

    The following theorem is analogous to the inertia bound \cite{Cvetkovic1971,Abiad2021} for the strong out-independence number:

\begin{theorem}\label{thm:inertia}
Let $G=(V,E)$ be a simple graph with minimum degree $\ge2$, and let $\mathcal{G}=(\mathcal{V},\mathcal{E})$ be its non-backtracking graph.   Then,

\begin{equation*}
    \alpha_{s\text{-}out}(\mathcal{G})\le\#\biggl\{i:s_i(\mathcal{L})\le\sqrt{\frac{\delta}{\delta-1}}  \biggr\}
\end{equation*}and
\begin{equation*}
    \alpha_{s\text{-}out}(\mathcal{G})\le\#\biggl\{i:s_i(\mathcal{L})\ge\sqrt{\frac{\Delta}{\Delta-1}}  \biggr\}.
\end{equation*}
\end{theorem}

\begin{proof}
Let $S\subset \mathcal{V}$ be a maximal strong out-independent set, and let  $$\mathcal{X}:=\mathrm{span}_{\R}\{\mathbf{1}_{[v,w]}:[v,w]\in S\}\subset \R^{2M},$$ 
where $\mathbf{1}_{[v,w]}\in \R^{2M}$ is the characteristic vector of $[v,w]$, therefore its entries are
\begin{equation*}
    (\mathbf{1}_{[v,w]})_j=\begin{cases}
        1, & \text{if }[w,j]=e_j\\
        0, & \text{otherwise.}
    \end{cases}
\end{equation*}

Then, $\dim_\R\mathcal{X}=|S|=\alpha_{s\text{-}out}(\mathcal{G})$, and we shall simply denote it by $\alpha$ in this  proof.\newline Now, note that the matrix $\mathcal{L}^\top\mathcal{L}$ satisfies
\[(\mathcal{L}^\top\mathcal{L})_{[v,w],[y,v]}=(\mathcal{L}^\top\mathcal{L})_{[y,v],[v,w]}=-\frac{1}{\deg v-1},\]

\[(\mathcal{L}^\top\mathcal{L})_{[v,w],[v,w]}
=\frac{\deg v}{\deg v-1},\]
and for any $[v,w]\ne[v,w']$ in $\mathcal{V}$,
\[(\mathcal{L}^\top\mathcal{L})_{[v,w],[v,w']}=(\mathcal{L}^\top\mathcal{L})_{[v,w'],[v,w]}
=\frac{\deg v-2}{(\deg v-1)^2}.\]
All other entries of  $\mathcal{L}^\top\mathcal{L}$ are zero. Therefore, $(\mathcal{L}^\top\mathcal{L})_{[v,w],[y,z]}$ is non-zero if and only if one of the following three conditions is satisfied:
\begin{itemize}
    \item $[v,w]$ and $[y,z]$ are out-adjacent,
    \item $[v,w]$ and $[y,z]$ are adjacent, or
    \item $[v,w]=[y,z]$.
\end{itemize}
  Since $S$ is a strong out-independent set, one can easily check that, 
for any $f\in\mathcal{X}\setminus\{\mathbf{0}\}$,
\[\frac{\langle \mathcal{L}^\top\mathcal{L}f,f\rangle}{\langle f,f\rangle }=\frac{\sum_{[v,w]\in S}\frac{\deg v}{\deg v-1}\cdot f([v,w])^2}{\sum_{[v,w]\in S}f([v,w])^2},
\]
which implies that
\[\frac{\Delta}{\Delta-1}\le\min_{[v,w]\in S}\,\frac{\deg v}{\deg v-1}\le\frac{\langle \mathcal{L}^\top\mathcal{L}f,f\rangle}{\langle f,f\rangle }\le \max_{[v,w]\in S}\,\frac{\deg v}{\deg v-1}\le \frac{\delta}{\delta-1}.\]
Since the eigenvalues of the positive semi-definite matrix $\mathcal{L}^\top\mathcal{L}$ are given by \[0=s_1^2(\mathcal{L})\le s_2^2(\mathcal{L})\le\ldots\le s_{2M}^2(\mathcal{L}),\]
we can apply the min-max principle to derive 
\[s_\alpha^2(\mathcal{L})\le\max_{f\in \mathcal{X}\setminus\{\mathbf{0}\}}\frac{\langle \mathcal{L}^\top\mathcal{L}f,f\rangle}{\langle f,f\rangle }\le  \frac{\delta}{\delta-1}\]
and  
\[s_{2M-\alpha+1}^2(\mathcal{L})\ge\min_{f\in \mathcal{X}\setminus\{\mathbf{0}\}}\frac{\langle \mathcal{L}^\top\mathcal{L}f,f\rangle}{\langle f,f\rangle }\ge  \frac{\Delta}{\Delta-1}.\]
This proves the claim.
\end{proof}

As a generalization of Theorem \ref{thm:inertia}, we have: 

\begin{theorem}
Let $G=(V,E)$ be a simple graph with minimum degree $\ge2$, and let $\mathcal{G}=(\mathcal{V},\mathcal{E})$ be its non-backtracking graph.   Then, for any $a\in\R$, 
\begin{equation*}
    \alpha_{s\text{-}out}(\mathcal{G})\le\#\biggl\{i:s_i(a\,\mathrm{Id}-\mathcal{L})\le\sqrt{(a-1)^2+\frac{1}{\delta-1}}\biggr\}
\end{equation*}and
\begin{equation*}
    \alpha_{s\text{-}out}(\mathcal{G})\le\#\biggl\{i:s_i(a\,\mathrm{Id}-\mathcal{L})\ge\sqrt{(a-1)^2+\frac{1}{\Delta-1}}\biggr\}.
\end{equation*}
\end{theorem}

\begin{proof}
    Analogous to that of Theorem \ref{thm:inertia}.
\end{proof}

For $a=1$, the above theorem can be proved also for $\alpha_{out}(\mathcal{G})$.

\begin{theorem}\label{thm:out-indepen}
Let $G=(V,E)$ be a simple graph with minimum degree $\ge2$, and let $\mathcal{G}=(\mathcal{V},\mathcal{E})$ be its non-backtracking graph.   Then, 
\begin{equation*}
    \alpha_{out}(\mathcal{G})\le\#\biggl\{i:s_i(a\,\mathrm{Id}-\mathcal{L})\le\sqrt{\frac{1}{\delta-1}}\biggr\}
\end{equation*}and
\begin{equation*}
    \alpha_{out}(\mathcal{G})\le\#\biggl\{i:s_i(a\,\mathrm{Id}-\mathcal{L})\ge\sqrt{\frac{1}{\Delta-1}}\biggr\}.
\end{equation*}
\end{theorem}

\begin{proof}

The proof is a slight modification of the proof for  Theorem \ref{thm:inertia}. 
The only difference is that the matrix 
$(\mathrm{Id}-\mathcal{L}^\top)(\mathrm{Id}-\mathcal{L})=(\mathcal{D}^{-1}\mathcal{A})^\top\mathcal{D}^{-1}\mathcal{A}$ 
 satisfies
\[((\mathcal{D}^{-1}\mathcal{A})^\top\mathcal{D}^{-1}\mathcal{A})_{[v,w],[v,w]}
=\frac{1}{\deg v-1}\]
and, for any $[v,w]\ne[v,w']$ in $\mathcal{V}$,
\[((\mathcal{D}^{-1}\mathcal{A})^\top\mathcal{D}^{-1}\mathcal{A})_{[v,w],[v,w']}=(\mathcal{L}^\top\mathcal{L})_{[v,w'],[v,w]}
=\frac{\deg v-2}{(\deg v-1)^2},\]
while all other entries are zero. Hence, $((\mathcal{D}^{-1}\mathcal{A})^\top\mathcal{D}^{-1}\mathcal{A})_{[v,w],[y,z]}$ is non-zero if and only if either $[v,w]$ and $[y,z]$ are out-adjacent, 
 or $[v,w]=[y,z]$. We can therefore use the out-independence relation  instead of the strong out-independence relation, following the proof of Theorem \ref{thm:inertia}. 
\end{proof}

\section{Spectral gap from a real number}\label{section:gapR}

In this section, for convenience, we assume that the non-backtracking graph $\mathcal{G}=(\mathcal{V},\mathcal{E})$ of $G=(V,E)$ is connected. That is, we assume that no connected component of $G$ is a cycle graph. We prove some bounds on the moduli of the eigenvalues of $\mathcal{L}$, and we relate them to the singular values that we investigated in the previous section. Before, we prove a preliminary lemma.

\begin{definition}
    Given a set $X\subseteq \mathbb{C}^{2M}$, we let
    $$X^{\bot_P}:=\{f\in\mathbb{C}^{2M}:(f,g)_P=0,\,\forall g\in X\}.$$ 
\end{definition}

\begin{lemma}\label{lemma:a}
    Let $(\lambda,f)$ and $(\mu,g)$ be two eigenpairs of $\mathcal{L}$ such that $\bar\lambda\ne\mu$.  Then, $(f,g)_P=0$.
\end{lemma}
\begin{proof}
    By Theorem 3.9 in \cite{NB-Laplacian}, $\mathcal{L}$ is self-adjoint with respect to the $P$-product. Together with the assumption that $(\lambda,f)$ and $(\mu,g)$ are eigenpairs, this implies that
\[\bar\lambda( f,g)_P=(\lambda f,g)_P=(\mathcal{L}f,g)_P=(f,\mathcal{L}g)_P=(f,\mu g)_P=\mu(f, g)_P.\]
This implies that $(\bar\lambda-\mu)(f, g)_P=0$, and therefore, since $\bar\lambda\ne\mu$, $(f,g)_P=0$.
\end{proof}

\begin{remark}
If we take $(\mu, g)=(\lambda, f)$ in Lemma \ref{lemma:a}, we obtain that $(f,f)_P=0$ when $\lambda$ is not  real. Hence, the lemma generalizes the third statement of Theorem 3.9 in \cite{NB-Laplacian}.
\end{remark}

Lemma \ref{lemma:a} allows us to prove the following
\begin{proposition}\label{prop:estimate-1}
Let $\lambda_1,\ldots,\lambda_k$ be $k< 2M$ eigenvalues of $\mathcal{L}$, and let $f_1,\ldots,f_k$ be corresponding eigenfunctions. Then, for any eigenvalue $\mu\in\sigma(\mathcal{L})\setminus\{\lambda_1,\ldots,\lambda_k,\bar\lambda_1,\ldots,\bar\lambda_k\}$,  \[\min\limits_{f\in\mathrm{span}(f_1,\ldots,f_k,\bar f_1,\ldots,\bar f_k)^{\bot_P}\setminus\{\mathbf{0}\}}\frac{\|\mathcal{L}f\|_2}{\|f\|_2}\le |\mu|\le \max\limits_{f\in\mathrm{span}(f_1,\ldots,f_k,\bar f_1,\ldots,\bar f_k)^{\bot_P}\setminus\{\mathbf{0}\}}\frac{\|\mathcal{L}f\|_2}{\|f\|_2},\]
where $\|\cdot\|_2$ indicates the standard $l^2$-norm on $\mathbb{C}^{2M}$. 
\end{proposition}

\begin{proof}
 Since $\mathcal{L}$ is a real matrix, $\bar\lambda_1,\ldots,\bar\lambda_k$ are also eigenvalues of $\mathcal{L}$, and the functions $\bar f_1,\ldots,\bar f_k$ are corresponding  eigenfunctions. Now, fix an eigenfunction $g$ for $\mu$. By Lemma \ref{lemma:a}, $(f_i,g)_P=0$ and $(\bar f_i,g)_P=0$, for all $i=1,\ldots,k$. Therefore, $(f,g)_P=0$, for all $f\in \mathrm{span}(f_1,\ldots,f_k,\bar f_1,\ldots,\bar f_k)$, that is, $g\in \mathrm{span}(f_1,\ldots,f_k,\bar f_1,\ldots,\bar f_k)^{\bot_P}$. Since  $\|\mathcal{L}g\|_2=\|\mu g\|_2=|\mu|\cdot \|g\|_2$, we have 
\[\min\limits_{f\in\mathrm{span}(f_1,\ldots,f_k,\bar f_1,\ldots,\bar f_k)^{\bot_P}\setminus\{\mathbf{0}\}}\frac{\|\mathcal{L}f\|_2}{\|f\|_2}\le \frac{\|\mathcal{L}g\|_2}{\|g\|_2}=|\mu|\le \max\limits_{f\in\mathrm{span}(f_1,\ldots,f_k,\bar f_1,\ldots,\bar f_k)^{\bot_P}\setminus\{\mathbf{0}\}}\frac{\|\mathcal{L}f\|_2}{\|f\|_2}.\]
\end{proof}

In the same way, one can prove the following estimate on the distance between a set of eigenvalues of $\mathcal{L}$ and a given real number $a$.

\begin{proposition}\label{pro:estimate-a-L}
Let $\lambda_1,\ldots,\lambda_k$ be $k< 2M$ eigenvalues of $\mathcal{L}$, let $f_1,\ldots,f_k$ be corresponding eigenfunctions, and let $a\in\mathbb{R}$. Then, for any eigenvalue $\mu\in\sigma(\mathcal{L})\setminus\{\lambda_1,\ldots,\lambda_k,\bar\lambda_1,\ldots,\bar\lambda_k\}$, 
  \[\min\limits_{f\in\mathrm{span}(f_1,\ldots,f_k,\bar f_1,\ldots,\bar f_k)^{\bot_P}\setminus\{\mathbf{0}\}}\frac{\|(a\,\mathrm{Id}-\mathcal{L})f\|_2}{\|f\|_2}\le |\mu-a|\] 
and 
\[ |\mu-a|\le \max\limits_{f\in\mathrm{span}(f_1,\ldots,f_k,\bar f_1,\ldots,\bar f_k)^{\bot_P}\setminus\{\mathbf{0}\}}\frac{\|(a\,\mathrm{Id}-\mathcal{L})f\|_2}{\|f\|_2}.\]
\end{proposition}

\begin{proof}
    Analogous to the proof of Proposition \ref{prop:estimate-1}.
\end{proof}

Now, in Theorem \ref{thm:bound-2} below, we show that the modulus of every non-zero eigenvalue of $\mathcal{L}$ is bounded below by $s_2(\mathcal{L})$, and above by $s_{2M}(\mathcal{L})$.

\begin{theorem}\label{thm:bound-2}For every non-zero eigenvalue $\lambda$ of $\mathcal{L}$,
\[s_2(\mathcal{L}) \leq |\lambda|\leq s_{2M}(\mathcal{L}) .\]
\end{theorem}

\begin{proof}
We know that $0$ is an eigenvalue of $\mathcal{L}$, and that the constant function $\mathbf{1}_{\mathcal{V}}$ that maps all vertices of $\mathcal{G}$ to $1$ is a corresponding eigenfunction. Therefore, in Proposition \ref{prop:estimate-1} we can choose $k=1$, $\lambda_1=0$ and $f_1=\mathbf{1}_{\mathcal{V}}$ to infer that
\begin{equation*}
    |\lambda|\geq \min\limits_{f\in\mathrm{span}(\mathbf{1}_{\mathcal{V}})^{\bot_P}\setminus\{\mathbf{0}\}}\frac{\|\mathcal{L}f\|_2}{\|f\|_2}.
\end{equation*}Also, since $(f,\mathbf{1}_{\mathcal{V}})_P=0$ if and only if $\langle f,\mathbf{1}_{\mathcal{V}}\rangle=0$, we have that \[\min\limits_{f\in\mathrm{span}(\mathbf{1}_{\mathcal{V}})^{\bot_P}\setminus\{\mathbf{0}\}}\frac{\|\mathcal{L}f\|_2}{\|f\|_2}=\min\limits_{f\in\mathrm{span}(\mathbf{1}_{\mathcal{V}})^{\bot}\setminus\{\mathbf{0}\}}\frac{\|\mathcal{L}f\|_2}{\|f\|_2}.\]
Moreover, since $\mathcal{L}^T\mathcal{L}$ is positive semidefinite, 
\begin{align*}
\min\limits_{f\in\mathrm{span}(\mathbf{1}_{\mathcal{V}})^{\bot}\setminus\{\mathbf{0}\}}\frac{\|\mathcal{L}f\|_2}{\|f\|_2}&=\sqrt{\min\limits_{f\in\mathrm{span}(\mathbf{1}_{\mathcal{V}})^{\bot}\setminus\{\mathbf{0}\}}\frac{\|\mathcal{L}f\|_2^2}{\|f\|_2^2}}\\&=\sqrt{\min\limits_{f\in\mathbb{R}^{2M}\cap\mathrm{span}(\mathbf{1}_{\mathcal{V}})^{\bot}\setminus\{\mathbf{0}\}}\frac{\langle \mathcal{L}^\top\mathcal{L}f,f\rangle}{\langle f,f\rangle }}
\\&=\sqrt{\lambda_2(\mathcal{L}^\top\mathcal{L})}\\&=s_2(\mathcal{L}).
\end{align*}
Putting everything together, this proves the lower bound $|\lambda|\ge s_2(\mathcal{L})$. The proof of the upper bound $|\lambda|\le s_{2M}(\mathcal{L})$ is analogous.
\end{proof}

We now generalize Theorem \ref{thm:bound-2} as follows.

\begin{theorem}
For every non-zero eigenvalue $\lambda$ of $\mathcal{L}$ and for every $a\in\R$,
\[s_2(a\,\mathrm{Id}-\mathcal{L}) \le |\lambda-a|\le s_{2M}(a\,\mathrm{Id}-\mathcal{L}).\]
\end{theorem}

\begin{proof}
By Proposition \ref{pro:estimate-a-L}, it suffices to estimate the quantity 
\[\frac{\|(a\,\mathrm{Id}-\mathcal{L})f\|_2}{\|f\|_2}.\]
The rest of the proof is analogous to that of Theorem \ref{thm:bound-2}.
\end{proof}

\begin{remark}\label{remark:1-L} 
Not only the above theorem generalizes Theorem \ref{thm:bound-2} above, but also Theorem 4.1 in \cite{NB-Laplacian}, which states that the spectral gap from $1$ is larger or equal than $\frac{1}{\Delta-1}$. This follows, in particular, by taking $a=1$ and by showing that $s_2(\mathrm{Id}-\mathcal{L})=\frac{1}{\Delta-1}$, as follows.\newline
As shown in the proof of Theorem \ref{thm:out-indepen}, for any $f:\mathcal{V}\to\R$, we have 
\begin{align*}
&\langle (\mathcal{D}^{-1}\mathcal{A})^\top\mathcal{D}^{-1}\mathcal{A}f,f\rangle\\
&=\sum_{[v,w]\in\mathcal{V}}\frac{1}{\deg v-1}\cdot f^2([v,w])+\sum_{[v,w]\ne [v,w']\in\mathcal{V}}2\cdot \frac{\deg v-2}{(\deg v-1)^2}\cdot f([v,w])f([v,w'])
\\&= \sum_{[v,w]\in\mathcal{V}}\frac{1}{(\deg v-1)^2}\cdot f^2([v,w])+\sum_{v\in V}\left(\sum_{w\sim v}f([v,w])\right)^2
\\&\ge \sum_{[v,w]\in\mathcal{V}}\frac{1}{(\deg v-1)^2}\cdot f^2([v,w])\\ &\ge \frac{1}{(\Delta-1)^2}\langle f,f\rangle.
\end{align*}
Now, fix three distinct vertices $v,w,w'\in V$ such that $\deg v=\Delta$ and $\{v,w\},\{v,w'\}\in E$. Let then $f:V\to\R$ be such that $f([v,w]):=1$, $f([v,w']):=-1$, and it is zero otherwise. Then, $f\in \mathrm{span}(\mathbf{1}_{\mathcal{V}})^{\bot}\setminus\{\mathbf{0}\}$ and the above inequality reduces to an equality. 
Therefore, 
\[\lambda_2((\mathcal{D}^{-1}\mathcal{A})^\top\mathcal{D}^{-1}\mathcal{A})=\min\limits_{f\in\mathrm{span}(\mathbf{1}_{\mathcal{V}})^{\bot}\setminus\{\mathbf{0}\}}\frac{\langle \mathcal{L}^\top\mathcal{L}f,f\rangle}{\langle f,f\rangle }=\frac{1}{(\Delta-1)^2}.\]
Hence, $s_2(\mathrm{Id}-\mathcal{L})=\sqrt{\lambda_2((\mathcal{D}^{-1}\mathcal{A})^\top\mathcal{D}^{-1}\mathcal{A})}=\frac{1}{\Delta-1}$. 
\end{remark}

\section*{Acknowledgments}
In 2022, Raffaella Mulas and Leo Torres gave a course on non-backtracking operators of graphs at the Max Planck Institute for Mathematics in the Sciences. As part of the course material, and as a follow-up to their first joint article, they prepared a list of open questions on the non-backtracking Laplacian that led to many discussions and exchanging of ideas with some of the course attendees, and eventually to this paper. As such, we would particularly like to express our gratitude to Leo Torres for his indirect contribution to this work. We would also like to thank Florentin Münch, Jiaxi Nie and, in particular, Janis Keck for the helpful comments and interesting discussions. Last (and least), we are grateful to Conor Finn for helping us choosing a suggestive title which we hope will not be rejected.

\bibliographystyle{plain} 

\bibliography{NB}

\end{document}